\documentclass[a4paper]{article}
\usepackage{ifthen}
\def\style{preprint}
\ifthenelse{\equal{\style}{preprint}}{
\usepackage{dmo_preprint}
}{}
\usepackage{amsthm,amsmath,amsfonts,graphicx}
\usepackage[T1]{fontenc}
\usepackage{amssymb}

\def\Z{\mathbb{Z}}

\DeclareMathOperator{\supp}{supp}

\newtheorem{theorem}{Theorem}[section]
\newtheorem*{theorem*}{Theorem}
\newtheorem{lemma}[theorem]{Lemma}
\newtheorem{corollary}[theorem]{Corollary}

\newtheorem{definition}[theorem]{Definition}
\newtheorem{proposition}[theorem]{Proposition}
\newtheorem{observation}[theorem]{Observation}

\title{The Chromatic Polynomial of a Digraph}
\author{Winfried Hochst\"attler, Johanna Wiehe \\ FernUniversit\"at in Hagen, Germany}
\date{}
\begin{document}
\ifthenelse{\equal{\style}{preprint}}{ \DMOmathsubject{05C20,05C31,05C21,05C15}
  \DMOkeywords{flow polynomial, dichromatic number, totally cyclic subdigraphs, face lattice, chromatic polynomial} \DMOtitle{063.19}{The Chromatic Polynomial of a Digraph} {Winfried Hochst\"attler, Johanna
    Wiehe}{winfried.hochstaettler@fernuni-hagen.de,
    johanna.wiehe@fernuni-hagen.de} }{}
\maketitle

\begin{abstract}
An acyclic coloring of a digraph as defined by Neumann-Lara is a vertex-coloring such that no monochromatic directed cycles occur.\\
Counting the number of such colorings with $k$ colors can be done by counting so-called Neumann-Lara-coflows (NL-coflows), which build a polynomial in $k$. We will present a representation of this polynomial using totally cyclic subdigraphs, which form a graded poset $Q$.\\
Furthermore we will decompose our NL-coflow polynomial, which becomes the chromatic polynomial of a digraph by multiplication with the number of colors to the number of components, examining the special structure of the poset of totally cyclic subdigraphs with fixed underlying undirected graph. \\
This decomposition will confirm the equality of our chromatic polynomial of a digraph and the chromatic polynomial of the underlying undirected graph in the case of symmetric digraphs.
\end{abstract}

\section{Introduction}

The notion of classic graph coloring deals with finding the smallest integer $k$ such that the vertices of an undirected graph can be colored with $k$ colors, where no two adjacent vertices share the same color. The chromatic polynomial counts those proper colorings a graph admits, subject to the number of colors. Tutte developed a dual concept \cite{tuttecon}, namely his nowhere-zero flows (NZ-flows), which build a polynomial, the flow polynomial, too.

We turn our attention to directed graphs, or digraphs for short. In 1982 Neumann-Lara \cite{neumannlara} introduced the dichromatic number of a digraph $D$ as the smallest integer $k$ such that the vertices of $D$ can be colored with $k$ colors and each color class induces an acyclic digraph. This is reasonable generalization of the chromatic number since both numbers coincide in the symmetric case, where we have all arcs in both directions.

Moreover Neumann-Lara conjectured in 1985, that every orientation of a simple planar graph can be acyclically colored with two colors \cite{neumannlaraplanar}. Regarding the dichromatic number this is not the only conjecture remaining widely open. Up to some relaxations, for instance Mohar and Li \cite{mohar_li} affirmed the two-color-conjecture for planar digraphs of digirth four, it is known \cite{bokal}, that deciding whether an arbitrary digraph has dichromatic number at most two is NP-complete. 

Although some progress has been made according thresholds (see e.g.\ \cite{erdoes_gimbel_kratsch}), even the complete case seems to be quite hard. To our knowledge it is not known how many vertices suffice to build a tournament which has dichromatic number five \cite{neumannlaratournaments}.

Nevertheless, Ellis and Soukup determined \cite{cyclereversions} thresholds for the minimum number of cycles, where reversing their orientation yields a digraph resp. tournament that has dichromatic number at most two. 

Comparing the chromatic and the dichromatic number Erd\"os and Neumann-Lara conjectured \cite{erdoes_neumannlara} in 1979 that if the dichromatic number of a class of graphs is bounded, so is their chromatic number. While Mohar and Wu \cite{mohar_wu} considered the fractional chromatic number of linear programming proving a fractional version, this is another conjecture remaining unsolved.

With our work we hope to contribute to a better understanding of the dichromatic number. Hochst\"attler \cite{hochstarticle} developed a flow theory for the dichromatic number transferring Tutte's theory of NZ-flows from classic graph colorings. Together with Altenbokum \cite{althowie} we pursued this analogy by introducing algebraic Neumann-Lara-flows (NL-flows) as well as a polynomial counting these flows. The formula we derived contains the M\"obius function of a certain poset. Here, we will derive the values of the M\"obius function by showing that the poset correlates to the face lattice of a polyhedral cone.

Probably, the chromatic polynomial of a graph is better known than the flow polynomial. Therefore, in this paper we consider the dual case of our NL-flow polynomial, the NL-coflow polynomial which equals the chromatic polynomial for the dichromatic number divided by the number of colors if the digraph is connected. We will present a representation using totally cyclic subdigraphs and decompose them to obtain an even simpler representation. In particular, it will suffice to consider certain subsets of edges of the underlying undirected graph.

Our notation is fairly standard and, if not explicitly defined, should follow the books of Bondy and Murty \cite{bondy} for digraphs and Beck and Sanyal \cite{BeckSanyal} for polyhedral geometry. Note that all our digraphs may have parallel and antiparallel arcs as well as loops if not explicitly excluded.

\section{Definitions and Tools}
Let $\mathcal{G}$ be a finite Abelian group and $D=(V,A)$ a digraph. Recall that a map \mbox{$f: A \longrightarrow \mathcal{G}$} is a flow in $D$, if it satisfies Kirchhoff's law of flow conservation
\begin{eqnarray}\label{flowcondition}
\sum_{a \in \partial^+(v)}f(a) =\sum_{a \in \partial^-(v)}f(a)
\end{eqnarray}
in every vertex $v\in V$, where $\partial^+(v)$ and $\partial^-(v)$ denote the set of outgoing resp. incoming arcs at $v$.\\
Analogously, a map $g: A \longrightarrow \mathcal{G}$ is a coflow in $D$, if it satisfies Kirchhoff's law for (weak) cycles $C \subseteq A$
\begin{eqnarray}\label{coflowcondition}
\sum_{a \in C^+}g(a)=\sum_{a \in C^-}g(a),
\end{eqnarray}
where $C^+$ and $C^-$ denote the set of arcs in $C$ that are traversed in forward resp. in backward direction.\\
Now let $n$ be the number of vertices, $m$ be the number of arcs and let $M$ denote the totally unimodular $(n \times m)$-incidence matrix of $D$. While condition (\ref{flowcondition}) is equivalent to the condition that the vector $f=(f(a_1),\ldots,f(a_m))^\top$ is an element of the null space of $M$, that is $Mf=0$, condition (\ref{coflowcondition}) is equivalent to the condition that the vector $g=(g(a_1),\ldots,g(a_m))$ is an element of the row space of $M$, that is $g=pM$, for some $(1 \times n)$-vector $p \in \mathcal{G}^{n}$. 

\begin{definition}
A digraph $D=(V,A)$ is called \emph{totally cyclic}, if every component is strongly connected. A \emph{feedback arc set} of a digraph is a set $S \subseteq A$ such that $D-S$ is acyclic. 
\end{definition}

\begin{definition}
Let $D=(V,A)$ be a digraph and $\mathcal{G}$ a finite Abelian group. An \emph{NL-$\mathcal{G}$-coflow} in $D$ is a coflow $g: A \longrightarrow \mathcal{G}$ in $D$ whose support contains a feedback arc set. For $k \in \mathbb{Z}$ and $\mathcal{G}=\Z$, a coflow $g$ is an \emph{NL-$k$-coflow}, if
\begin{equation*} 
g(a) \in \lbrace 0,\pm 1,\ldots,\pm (k-1) \rbrace \text{ , for all } a \in A,
\end{equation*}
such that its support contains a feedback arc set.
\end{definition}

In order to develop a closed formula for the number of NL-$\mathcal{G}$-coflows we use a kind of generalization of the well-known inclusion-exclusion formula, the M\"obius inversion.

\begin{definition}[see e.g.\ \cite{aigner}]
Let $(P,\leq)$ be a finite poset, then the {\em M\"obius function} is defined as follows
\begin{equation*} 
\mu: P \times P \rightarrow \mathbb{Z},\; \mu(x,y):= 
\begin{cases} 0 &\; ,\text{ if } x\nleq y \\
			  1 &\; ,\text{ if }  x=y\\
			  -\sum_{x \leq z < y} \mu(x,z) &\; ,\text{ otherwise }.
\end{cases}
\end{equation*}
\end{definition}

\begin{proposition}[see \cite{aigner}, \cite{BeckSanyal}]\label{moebius inversion}
Let $(P,\leq)$ be a finite poset, $f,g : P \longrightarrow \mathbb{K}$ functions and $\mu$ the M\"obius function. Then the following equivalence holds
\begin{equation*} 
f(x)=\sum_{y \geq x}g(y) , \text{ for all } x \in P \Longleftrightarrow g(x)=\sum_{y \geq x}\mu(x,y)f(y), \text{ for all } x \in P.
\end{equation*}
\end{proposition}
With this so called \emph{M\"obius inversion from above} it will
suffice to compute the number of $\mathcal{G}$-coflows in some given minors $B$, which is $|\mathcal{G}|^{rk(B)}$, where $rk(B)$ is the rank of the incidence matrix of $D[B]$ which equals \mbox{$|V(B)|-c(B)$}, i.e.\ the number of vertices minus the number of connected components of $D[B]$.

\section{The NL-Coflow Polynomial}

In this chapter we will define the NL-coflow polynomial, which counts the number of NL-$\mathcal{G}$-coflows, using M\"obius inversion. Therefor we need a specific partially ordered set. The following poset $(\mathcal{C},\geq)$ with
\begin{equation*}
\mathcal{C}:= \big \lbrace A /  C \; \vert \; \exists \; C_1,\ldots,C_r \text{ directed cycles, such that } C=\bigcup_{i=1}^r C_i \big \rbrace 
\end{equation*}
and
\begin{equation*}
A/\bigcup_{j\in J}C_j \geq A/\bigcup_{i \in I}C_i :\Leftrightarrow \bigcup_{j \in J} C_j \subseteq \bigcup_{i \in I} C_i,
\end{equation*} 
will serve our purpose. Note that in case $D$ is strongly connected,
$A$ is the unique minimum of this poset.

\begin{definition}\label{defcoflow}
Let $D=(V,A)$ be a digraph and $\mu$ the M\"obius function of $\mathcal{C}$. Then the {\em NL-Coflow Polynomial} of $D$ is defined as
\begin{equation*}
\psi_{NL}^D(x):= \sum_{Y \in \mathcal{C}}\mu(A,Y)x^{rk(Y)}.
\end{equation*}
\end{definition}
The dual version of Theorem 3.5 in \cite{althowie} reveals the following.

\begin{theorem}\label{NLcoflowTheo}
The number of NL-$\mathcal{G}$-coflows of a digraph $D$ depends only on the order $k$ of $\mathcal{G}$ and is given by $\psi_{NL}^D(k)$.
\end{theorem}

\begin{proof}
Using Proposition \ref{moebius inversion} with 
$f_k,g_k: \mathcal{C} \rightarrow \Z$, such that $f_k(Y)$ indicates all $\mathcal{G}$-coflows
and $g_k(Y)$ all NL-$\mathcal{G}$-coflows in $D[Y]$,
it suffices to show that
\begin{eqnarray}\label{moebiusvor}
f_k(Z)=\sum_{\substack{Y \leq Z\\Y \in \mathcal{C}}}g_k(Y)
\end{eqnarray}
holds for all $Z \in \mathcal{C}$. Then we obtain
\begin{eqnarray*}
\psi_{NL}^D(k)=g_k(A) = \sum_{\substack{Y \leq A\\Y\in \mathcal{C}}} \mu(A,Y)f_k(Y)
= \sum_{Y \in \mathcal{C}}\mu(A,Y)k^{rk(Y)},
\end{eqnarray*}
since the number of $\mathcal{G}$-coflows on $D[Y]$ is given by $k^{rk(Y)}$.\\
Concerning (\ref{moebiusvor}) let $Z \in \mathcal{C}$ and $\varphi$ be a $\mathcal{G}$-coflow on $D[Z]$. With $d$ we denote the number of directed cycles in $D[Z]$ and set
\begin{equation*}
Y := Z / \bigcup_{i=1}^d \left\lbrace C_i \; \vert \; C_i \text{ is a directed cycle in }D[Z]\text{ and } \forall c \in C_i : \varphi(c)=0 \right\rbrace.
\end{equation*}
Then clearly $Y \in \mathcal{C}$ and $\varphi|_{Y}$ is an NL-$\mathcal{G}$-coflow in $D[Y]$. \\
The other direction is obvious since every NL-$\mathcal{G}$-coflow $g$ on $D[Y]$ with $Y \in \mathcal{C}$ can be extended to a $\mathcal{G}$-coflow $\tilde{g}$ on $D[Z]$, setting $\tilde{g}(a):= 0_{\mathcal{G}}$ for all $a \in Z\setminus Y$. 
\end{proof}

\subsection{Totally Cyclic Subdigraphs}

Since many unions of directed cycles determine the same strongly connected subdigraph it suffices to consider all totally cyclic subdigraphs. Those form a graded poset and the M\"obius function simply alternates.

\begin{lemma}\label{lemma:cone}
Let $D=(V,A)$ be a digraph. The poset
$$ Q:=\lbrace B \subseteq A: D[B] \text{ is totally cyclic subdigraph of } D\rbrace,$$
ordered by inclusion is a graded lattice with rank function $rk_Q$ and its M\"obius function alternates in the following fashion:
$$ \mu_Q(\emptyset,B)=(-1)^{rk_Q(B)}.$$
\end{lemma}

\begin{proof}
Let $M$ be the totally unimodular $(n\times m)-$incidence matrix of $D$. Then 
$$\begin{pmatrix}
M\\-M\\-Id_{m}
\end{pmatrix}x\leq 0$$
describes a polyhedral cone $C$. Since $M$ is totally unimodular all extreme rays are spanned by integral points. \\
Let $\mathcal{F}(C)$ be the face lattice of $C$. By $\mathbf{0}$ we denote the $m-$dimensional all-zero-vector, i.e.\ the only vertex of $C$ and define an order isomorphism between $Q$ and $\mathcal{F}(C)\setminus \lbrace \emptyset\rbrace$. For a nonempty face $F$ of $C$ we set 
$$B_F:= \supp(x)=\lbrace 1\leq i \leq m: x_i>0\rbrace$$ 
for some $x\in F$. Note that $B_\mathbf{0}=\emptyset$. Since $Mx=0$ Kirchhoff's flow condition is fulfilled and $x$ is a NZ-flow on $D[B_F]$.  Thus $D[B_F]$ is totally cyclic. For the other direction take a totally cyclic subdigraph $D[B]$ on $B\subseteq A$. Then there exists a NZ-flow $x$ on $D[B]$, for instance by assigning $x_a$ to the number of directed cycles $a\in B$ is in. By assigning the remaining arcs to zero we extend $x$ to $A$ and set
$$F_B:=\lbrace y: My=0, \supp(y)=\supp(x)\rbrace=\lbrace y: My=0, y_i=0 \Leftrightarrow x_i\notin B, \forall i \rbrace.$$
Note that $F_\emptyset=\mathbf{0}$. Clearly, $F_B$ is a nonempty face of $C$.\\
Now let $F,G$ be nonempty faces of $C$ and $x\in F, y\in G$. Then $F\subseteq G$ if and only if $\supp(x)\subseteq \supp(y)$, thus the poset $Q$ is isomorphic to $\mathcal{F}(C)\setminus \lbrace \emptyset \rbrace$ and 
$$rk_Q(B)=\begin{cases} 0 &\text{ , if } B=\emptyset\\ rk_{\mathcal{F}}(F_B)-1 &\text{ , otherwise}.\end{cases}$$ 
Theorem 3.5.1 and Corollary 3.3.3 in \cite{BeckSanyal} yield for $Y\subseteq X\in Q$
\begin{flalign*}
\mu_Q(Y,X)&=\mu_{C}(F_Y,F_X)=(-1)^{\dim F_X- \dim F_Y}\\
&=(-1)^{rk_C(F_X)-1-(rk_C(F_Y)-1)}=(-1)^{rk_Q(X)-rk_Q(Y)},
\end{flalign*}
and in particular 
$$ \mu_Q(\emptyset,X)=(-1)^{rk_Q(X)}.$$
\end{proof}

The \emph{corank}, also known as \emph{circuit rank}, of a graph is the number of independent cycles, i.e.\ the number of edges minus the number of vertices plus the number of connected components. The corank of a graph $G=(V,E)$ coincides with the number of edges that have to be removed to break all cycles and is denoted by $cr(G)$. For a subset $X\subseteq E$ of edges we denote by $cr(X)$ the corank of $G[X]$. For a set of arcs $X\subseteq A$ of a digraph $D=(V,A)$ we denote by $\underline{X}$ the set of edges of the underlying undirected graph $G[D[X]]$.

\begin{lemma}\label{rk_Q}
Let $X\in Q$ and $\delta(X)$ the number of digons in $D[X]$, then
$$ rk_Q(X)=cr(\underline{X})+\delta(X).$$
\end{lemma}

\begin{proof}
We proceed by induction over $rk_Q(X)$. If $rk_Q(X)=0$ then $X=\emptyset$ and $cr(\emptyset)+\delta(\emptyset)=0$. If there are no digons $D[X]$ is a bridgeless totally cyclic orientation of $G[\underline{X}]$. As the corank equals the number of independent cycles $rk_Q(X)=cr(\underline{X})$. Now let $rk_Q(X)>0$ and $d \in \underline{X}$ a digon in $D[X]$. This yields the following two cases:
\begin{enumerate}
\item[1.] $d$ is a bridge in $G[\underline{X}]$. Then 
$$rk_Q(X)=rk_Q(X\setminus \lbrace d \rbrace)+1\overset{IH}{=}cr(\underline{X\setminus \lbrace d\rbrace})+\delta(X\setminus \lbrace d \rbrace)+1=cr(\underline{X})+\delta(X).$$
\item[2.] $d$ is not a bridge in $G[\underline{X}]$. Then there exists an arc $\overrightarrow{d}\in X$ such that $X\setminus\lbrace \overrightarrow{d}\rbrace\in Q$ and
$$rk_Q(X)=rk_Q(X\setminus \lbrace \overrightarrow{d} \rbrace)+1\overset{IH}{=}cr(\underline{X\setminus \lbrace \overrightarrow{d} \rbrace})+\delta(X\setminus \lbrace \overrightarrow{d} \rbrace)+1=cr(\underline{X})+\delta(X).$$
\end{enumerate}
\end{proof}

The following is immediate from the previous two lemmas:

\begin{corollary}\label{Cor:NL-coflow}
$$\psi_{NL}^D(k)= \sum_{X\in Q} (-1)^{cr(\underline{X})+\delta(X)}k^{rk(D/X)}.$$
\end{corollary}

\section{Decomposing the NL-Coflow Polynomial}

In this section we will examine the special structure of the poset of totally cyclic subdigraphs with a fixed underlying undirected graph. This will lead to a more compact representation of the NL-coflow polynomial and therefore also of the chromatic polynomial of a digraph.\\

Let $D=(V,A)$ be a digraph. For a subset of edges $Y\subseteq E(G[D])$ we denote
$$ Q(Y):=\lbrace X\in Q: \underline{X}=Y\rbrace$$
the elements of $Q$ with fixed support $Y$. Note that $Q(Y)$ does not need to be a lattice. However, once $Y$ is fixed there is a unique maximal element $\overline{Y}\in Q(Y)$ with $X\subseteq \overline{Y}$ for all $X\in Q(Y)$, namely 
$$\overline{Y}:=\bigvee_{\substack{X\in Q(Y)}}X.$$
Since there can be several totally cyclic orientations of $Y$ a unique minimal element does not necessarily exist. By $\mathcal{O}(Y)$ we denote the set of inclusionwise minimal totally cyclic orientations of $Y$. Note that the only digons in $\mathcal{O}(Y)$ are bridges in $G[D[Y]]$.

A first observation shows that a bridge in the underlying undirected graph increases the rank of a given totally cyclic subdigraph by exactly one. Let $\beta(Y)$ denote the set of bridges in $G[Y]$.

\begin{observation}\label{obs:bridges}
Let $D=(V,A)$ be a digraph and $G=(V,E)$ its underlying undirected graph and $Y\subseteq E$. Then
$$ \sum_{\substack{ X \in Q(Y)}} \mu_Q(\emptyset,X)=(-1)^{|\beta(Y)|}\sum_{\substack{X \in Q(Y\setminus \beta(Y)) }} \mu_Q(\emptyset,X).$$
\end{observation}

\begin{proof}
If there are no bridges in $G[Y]$ there is nothing to prove. Now, let $\overrightarrow{b},\overleftarrow{b}\in A$ be a digon such that its underlying edge $b\in Y\subseteq E$ is a bridge in $G$. If there exist other arcs in $D$ with underlying edge $b$ we can reduce this case to the case where $\overrightarrow{b},\overleftarrow{b}$ are the only arcs on $b$ as follows. Removing any arc $a\in A\setminus \lbrace \overrightarrow{b},\overleftarrow{b} \rbrace$ with underlying edge $b$ yields a totally cyclic subdigraph with $rk_Q(X\setminus \lbrace a\rbrace)=rk_Q(X)-1$ for $X\in Q(Y)$. Without loss of generality let $a$ be parallel to $\overrightarrow{b}$. Then also $D[X\setminus \lbrace\overrightarrow{b}\rbrace]$ is totally cyclic and $rk_Q(X\setminus \lbrace \overrightarrow{b}\rbrace)=rk_Q(X)-1$. Hence
$$ \sum_{X\in Q(Y)}\mu_Q(\emptyset,X)=\sum_{\substack{X\in Q(Y)\\ a\notin X}}\mu_Q(\emptyset,X).$$
Now suppose $\overrightarrow{b},\overleftarrow{b}$ are the only arcs on $b$. Then neither $D[Y]-\overrightarrow{b}$ nor $D[Y]-\overleftarrow{b}$ is totally cyclic. Thus
\begin{flalign*}
  \mu_Q(\emptyset,X)= (-1)^{rk_Q(X)}= (-1)^{rk_Q(X\setminus \lbrace \overrightarrow{b},\overleftarrow{b} \rbrace)+1}=- \mu_Q(\emptyset,X)
\end{flalign*}
holds for every $X\in Q(Y)$. Furthermore $X,Z\in Q(Y)$ if and only if \mbox{$X\setminus \lbrace \overrightarrow{b},\overleftarrow{b}\rbrace$}, $Z\setminus \lbrace \overrightarrow{b},\overleftarrow{b}\rbrace \in Q(Y\setminus \lbrace b\rbrace)$ and 
$$X\subseteq Z \Leftrightarrow X\setminus \lbrace \overrightarrow{b},\overleftarrow{b}\rbrace \subseteq Z\setminus \lbrace \overrightarrow{b},\overleftarrow{b}\rbrace.$$
\end{proof}

For now we can neglect the existence of bridges.

\begin{observation}\label{obs:above}
Let $D=(V,A)$ be a digraph and let $\emptyset \neq Y\subseteq E(D)$ such that $G[Y]$ has no bridges. For any $Z\in Q(Y)\setminus \lbrace \overline{Y} \rbrace$ 
$$ \sum_{\substack{X\in Q(Y)\\X\supseteq Z}}\mu_Q(\emptyset,X)=0$$
holds.
\end{observation}

\begin{proof}
Let $d:=|\overline{Y}\setminus Z|$. Then $d>0$ and since $G[Y]$ is bridgeless
$$ \sum_{\substack{X\in Q(Y)\\X\supseteq Z}}\mu_Q(\emptyset,X)=\sum_{i=0}^{d}\binom{d}{i}(-1)^{rk_Q(Z)+i}=0.$$
\end{proof}

Using the principle of inclusion and exclusion we can exploit this observation in order to compress some parts of $Q$.

\begin{observation}\label{atomsum}
Let $D=(V,A)$ be a digraph and $\emptyset \neq Y\subseteq E(D)$ such that $G[Y]$ is bridgeless. Let $\mathcal{O}(Y)=\lbrace Z_1,\ldots,Z_k \rbrace$. If $\bigvee_{i\in I}Z_i\neq \overline{Y}$ for any $\emptyset \neq I\subseteq [k]$, then
$$ \sum_{\substack{X\in Q(Y)\\ X\supseteq Z_i\\ i\in [k]}}\mu_Q(\emptyset,X)=0.$$
\end{observation}

\begin{proof}
By the principle of inclusion and exclusion and by Observation \ref{obs:above}
$$\sum_{\substack{X \in Q(Y)\\ X\supseteq Z_i\\ i\in [k]}}\mu_Q(\emptyset,X)= \sum_{\emptyset\neq I \subseteq [k]}(-1)^{|I|+1}\sum_{\substack{X\in Q(Y)\\ X\supseteq Z_I}}\mu_Q(\emptyset,X)=0.$$
\end{proof}

\begin{lemma}\label{zero}
Let $D=(V,A)$ be a digraph such that its underlying undirected graph is bridgeless and let $\emptyset\neq Z\in Q$. If there exists a digon $\overrightarrow{a},\overleftarrow{a}$ in $D[Z]$ such that either $D[Z]-\overrightarrow{a}$ or $D[Z]-\overleftarrow{a}$ is not totally cyclic, then
$$ \sum_{\substack{X\in Q(\underline{Z})\\ X\subseteq Z}}\mu_Q(\emptyset,X)=0.$$
\end{lemma} 

\begin{proof}
Without loss of generality let $D[Z]-\overrightarrow{a}$ be not totally cyclic. Then $\overrightarrow{a}\in X$ for all $X\in Q(\underline{Z})$ with $X\subseteq Z$. Let $\mathcal{O}(\underline{Z})=\lbrace Z_1,\ldots,Z_k \rbrace$. For $\emptyset \neq I\subseteq [k]$ we set $Z_I:=\bigvee_{i\in I}Z_i$. Then $\overleftarrow{a}\notin Z_I$ for every $I\subseteq [k]$ but $\overleftarrow{a}\in Z$ and thus $Z_I \neq Z$. By Observation \ref{atomsum}
$$ \sum_{\substack{X\in Q(\underline{Z})\\ X\subseteq Z}}\mu_Q(\emptyset,X)=\sum_{\substack{X \in Q(\underline{Z})\\ X\supseteq Z_i\\ i\in [k]}}\mu_Q(\emptyset,X)= 0.$$
\end{proof}

By $S(D)\subseteq A$ we denote the \emph{symmetric part} of a digraph $D=(V,A)$, that is the arcset of the maximal symmetric subdigraph of $D$. The following result shows that if the non-symmetric part is totally cyclic, the sum over all elements with fixed support reduces to $\pm 1$.

\begin{theorem}\label{symmetric}
Let $D=(V,A)$ be a digraph, $G=(V,E)$ its underlying undirected graph and $\emptyset \neq Y\subseteq E$. If $D[Y]-S(D[Y])$ is totally cyclic, then
$$ \sum_{\substack{X\in Q(Y)}}\mu_Q(\emptyset,X)=(-1)^{cr(Y\setminus \underline{S(D[Y])})+\delta(\overline{Y})}.$$
\end{theorem}

\begin{proof}
We proceed by induction on $\delta(\overline{Y})$. If $\overline{Y}$ has no digons Lemma \ref{rk_Q} yields
$$\sum_{\substack{X\in Q(Y)}}\mu_Q(\emptyset,X)=\mu_Q(\emptyset,\overline{Y})=(-1)^{cr(Y)}.$$
Now let $\delta(\overline{Y})>0$ and $\overrightarrow{a},\overleftarrow{a}\in \overline{Y}$ be a digon such that its underlying edge $a\in Y$ is not a bridge in $G[Y]$. If all the digons are bridges we can directly apply Observation \ref{obs:bridges} proving the claim. \\
Since $a\in S(D[Y])$ and by assumption both $D[Y]-S(D[Y])+\overleftarrow{a}$ and $D[Y]-S(D[Y])+\overrightarrow{a}$ are totally cyclic, so are $D[\overline{Y}]-\overrightarrow{a}$ and $D[\overline{Y}]-\overleftarrow{a}$. We decompose the elements $X\in Q(Y)$ into two parts:
$$\overrightarrow{X}:=\lbrace X\in Q(Y): \exists x\in \mathcal{O}(Y):x\subseteq X, \overrightarrow{a}\in x\rbrace$$
and
$$\overleftarrow{X}:=\lbrace X\in Q(Y): \exists x\in \mathcal{O}(Y):x\subseteq X, \overleftarrow{a}\in x\rbrace.$$
As a result
$$ \sum_{\substack{X\in Q(Y)}}\mu_Q(\emptyset,X)=\sum_{X\in \overrightarrow{X}}\mu_Q(\emptyset,X)+\sum_{X\in \overleftarrow{X}}\mu_Q(\emptyset,X)-\sum_{\substack{X\in \overrightarrow{X}\cap \overleftarrow{X}}}\mu_Q(\emptyset,X)$$
and by Observation \ref{atomsum} the first two terms equal zero.\\
Now we will prove that the poset $(\overrightarrow{X}\cap \overleftarrow{X},\subseteq)$ is isomorphic to $(Q(Y\setminus \lbrace a \rbrace),\subseteq)$.\\
Let $X,Z\in \overrightarrow{X}\cap \overleftarrow{X}$. Then both $D[X]- \overrightarrow{a}$ and $D[X]-\overleftarrow{a}$ are totally cyclic and so $X\setminus \lbrace \overrightarrow{a},\overleftarrow{a}\rbrace, Z\setminus \lbrace \overrightarrow{a},\overleftarrow{a} \rbrace\in  Q(Y\setminus \lbrace a \rbrace)$. Clearly, 
$$ X\subseteq Z \Leftrightarrow X\setminus \lbrace \overrightarrow{a},\overleftarrow{a}\rbrace \subseteq Z\setminus \lbrace \overrightarrow{a},\overleftarrow{a}\rbrace$$
and since $a$ is not a bridge and $\overrightarrow{a},\overleftarrow{a}\in X$
$$ rk_Q(X)=rk_Q(X\setminus \lbrace \overrightarrow{a},\overleftarrow{a}\rbrace)+2.$$
Using the induction hypothesis we find
$$\sum_{\substack{X\in \overrightarrow{X}\cap \overleftarrow{X}}}\mu_Q(\emptyset,X)=\sum_{\substack{X\in Q(Y\setminus \lbrace a\rbrace)}}\mu_Q(\emptyset,X)= (-1)^{cr(Y\setminus \underline{S(D[Y])})+\delta(\overline{Y})-1}$$
proving the theorem.
\end{proof}

Putting the two previous results together we find a new representation of the NL-coflow polynomial:

\begin{theorem}
Let $D=(V,A)$ be a digraph, $G=(V,E)$ its underlying undirected graph and 
$$Q^\prime:=\lbrace Y\subseteq E: D[Y] \text{ and } D[Y]-S(D[Y]) \text{ are totally cyclic}\rbrace.$$ 
Then
$$ \psi_{NL}(D,x)=\sum_{Y\in Q^\prime}(-1)^{cr(Y\setminus \underline{S(D[Y])})+\delta(\overline{Y})}x^{rk(D/Y)}.$$
\end{theorem}

\begin{proof}
Let $Y\subseteq E$ such that $D[Y]$ is totally cyclic but $D[Y]-S(D[Y])$ is not. Then there exists a digon $\overrightarrow{a},\overleftarrow{a}$ in $D[Y]$ such that either $D[Y]-\overrightarrow{a}$ or $D[Y]-\overleftarrow{a}$ is not totally cyclic. By Lemma \ref{zero} and Theorem \ref{symmetric} 
\begin{flalign*}
 \psi_{NL}(D,x)&=\sum_{Y\subseteq E}\sum_{\substack{X\in Q(Y)}}\mu_Q(\emptyset,X)x^{rk(D/X)}\\
 &=\sum_{Y\in Q^\prime}\sum_{\substack{X\in Q(Y)}}\mu_Q(\emptyset,X)x^{rk(D/X)}\\
 &= \sum_{Y\in Q^\prime}(-1)^{cr(Y\setminus \underline{S(D[Y])})+\delta(\overline{Y})}x^{rk(D/Y)}.
 \end{flalign*}
\end{proof}

\section{Symmetric Digraphs}

Considering symmetric digraphs $D=(V,A)$, it is obvious that the NL-coflow polynomial equals the chromatic polynomial $\chi(G,x)$ of the underlying undirected graph $G=(V,E)$ divided by the number of colors since both polynomials count the same objects. We will present an alternative proof of this fact, where the chromatic polynomial is represented by (see \cite{bondy})
\begin{equation*}
\chi (G,x)= \sum_{Y \subseteq E} (-1)^{|Y|}x^{\tilde{c}(Y)},
\end{equation*}
where $\tilde{c}(Y)$ denotes the number of components in the spanning subgraph $(V,Y)$.

\begin{corollary}
Let $D=(V,A)$ be a symmetric digraph, $G=(V,E)$ its underlying undirected graph and $\emptyset \neq Y\subseteq E$. Then:
$$ \sum_{\substack{X \in Q(Y)}}\mu_Q(\emptyset,X)=(-1)^{|Y|}.$$
\end{corollary}

\begin{proof}
Since the empty graph is totally cyclic Theorem \ref{symmetric} yields
$$ \sum_{\substack{X\in Q(Y)}}\mu_Q(\emptyset,X)= (-1)^{cr(\emptyset)+|\overline{Y}|}=(-1)^{|Y|}.$$
\end{proof}

\begin{corollary}
Let $D=(V,A)$ be a symmetric digraph and $G=(V,E)$ its underlying undirected graph. Then the following holds
\begin{equation*}
\psi_{NL}(D,x)= \chi(G,x)\cdot x^{-c(G)}.
\end{equation*} 
\end{corollary}

\begin{proof}
The previous Corollary yields
\begin{flalign*}
\psi_{NL}(D,x)&=\sum_{\emptyset \neq Y\subseteq E}\sum_{\substack{X\in Q(Y)}}\mu_Q(\emptyset,X)x^{rk(D/Y)}+\mu_Q(\emptyset,\emptyset)x^{rk(D)}\\
&=\sum_{\emptyset \neq Y\subseteq E}(-1)^{|Y|}x^{rk(D/Y)}+x^{rk(D)}.
\end{flalign*}
Since contracting does not change the number of components
\begin{flalign*}
rk(D/Y)&=|V(D/Y)|-c(D/Y)\\
&=c(D[Y])+|V(D)|-|V(Y)|-c(D/Y)=\tilde{c}(D[Y])-c(D)
\end{flalign*}
completes the proof.
\end{proof}

\bibliography{NL-flow}{}
\bibliographystyle{siam}

\end{document}